\documentclass{amsart}

\usepackage{amsthm,amsmath,amssymb}
\usepackage[numbers]{natbib}
\usepackage{graphicx}
\usepackage{bm}
\usepackage{color}

\definecolor{darkgreen}{rgb}{0,0.5,0}
\newtheorem{theorem}{Theorem}[section]

\newtheorem{lemma}[theorem]{Lemma}
\newtheorem{proposition}[theorem]{Proposition}
\theoremstyle{definition}
\newtheorem{definition}[theorem]{Definition}
\theoremstyle{remark}
\newtheorem{remark}[theorem]{Remark}
\newtheorem{example}[theorem]{Example}
\numberwithin{equation}{section}

\newcommand{\norm}[1]{\left\Vert#1\right\Vert}
\newcommand{\abs}[1]{\left\vert#1\right\vert}
\newcommand{\set}[1]{\left\{#1\right\}}
\newcommand{\R}{\mathbb R}
\newcommand{\N}{\mathbb N}

\newcommand{\bfx}{\bm{x}}
\newcommand{\bfzero}{\bm{0}}

\newcommand{\bfone}{\bm{1}}

\newcommand{\bfX}{\bm{X}}
\newcommand{\bfU}{\bm{U}}
\newcommand{\bfZ}{\bm{Z}}
\newcommand{\bfy}{\bm{y}}
\newcommand{\bfY}{\bm{Y}}

\newcommand{\bfV}{\bm{V}}

\newcommand{\bfzeta}{\bm{\zeta}}

\newcommand{\bfeta}{\bm{\eta}}
\newcommand{\bfxi}{\bm{\xi}}

\newcommand{\barE}{\bar E^-[0,1]}
\newcommand{\barC}{\bar C^-[0,1]}

\DeclareMathOperator{\PsiO}{\Psi}

\begin{document}

\title{On Max-Stable Processes and the Functional $D$-Norm}%
\author[Aulbach et al.]{Stefan Aulbach, Michael Falk and Martin Hofmann}

\address{\mbox{ } \newline University of W\"{u}rzburg \newline
Institute of Mathematics\newline  Emil-Fischer-Str. 30\newline 97074 W\"{u}rzburg\newline Germany\newline hofmann.martin@mathematik.uni-wuerzburg.de}
\thanks{The first author was supported by DFG Grant FA 262/4-1.}

\begin{abstract}
We introduce some mathematical framework for functional extreme
value theory and provide basic definitions and tools. In
particular we introduce a functional domain of attraction approach
for stochastic processes, which is more general than the usual one
based on weak convergence.

The distribution function $G$ of a continuous max-stable process
on $[0,1]$ is introduced and it is shown that $G$ can be
represented via a norm on functional space, called $D$-norm. This
is in complete accordance with the multivariate case and leads to
the definition of functional generalized Pareto distributions (GPD) $W$.
These satisfy $W=1+\log(G)$ in their upper tails, again in
complete accordance with the uni- or multivariate case.

Applying  this framework to copula processes we derive
characterizations of the domain of attraction condition for copula
processes in terms of tail equivalence with a functional GPD.

$\delta$-neighborhoods of a functional GPD are introduced and it
is shown that these are characterized by a polynomial rate of
convergence of functional extremes, which is well-known in the
multivariate case.
\end{abstract}


\maketitle

\section{Introduction}
Since the publication of the pathbreaking articles by  Pickands \cite{pick75} and Balkema and de Haan \cite{balh74} extreme value theory (EVT) has undergone a fundamental change. Instead of investigating the maxima in a set of observations, the focus is now on exceedances above a high threshold. The key result obtained in the above articles is the fact that the maximum of $n$ iid univariate observations, linearly standardized, follows an extreme value distribution (EVD) as $n$ increases if, and only if, the exceedances above an increasing threshold follow a generalized Pareto distribution (GPD). The multivariate analogon is due to Rootz\'{e}n and Tajvidi \cite{roott06}. For a recent account of multivariate EVT and GPD we refer to Falk et al. \cite{fahure10}.

An even more complex setup is
demanded by \textit{functional} EVT, which investigates maxima (taken pointwise) of
stochastic processes, as initiated by de Haan \cite{dehaan84} and de Haan and Pickands
\cite{dehap86}. We refer to de Haan and Ferreira \cite{dehaf06} for a detailed presentation
of up-to-date theory.

In particular de Haan and Lin \cite{dehal01}  considered weak convergence of the maximum of $n$ iid stochastic processes, linearly standardized, towards a max-stable process in the space $C[0,1]$ of continuous functions on $[0,1]$, topologized by the $\sup-$norm, and provided a domain of attraction condition. This condition consists, essentially, of the ordinary univariate weak convergence of the marginal maxima to a univariate EVD together with weak convergence of the corresponding copula process in functional space.

This is in accordance with multivariate EVT, where it is well-known that the maximum (taken
componentwise) of $n$ iid random vectors converges weakly to a multivariate EVT if, and
only if, this is true for the univariate maxima together with convergence of the
corresponding copulas (Deheuvels \cite{deheu78}, \cite{deheu84}, Galambos \cite{gal87}).

In the present paper we develop a framework for a functional domain of attraction theory, which is in even higher conformity with the multivariate case.

This paper is organized as follows. In Section
\ref{sec:extreme_value_processes} we introduce some mathematical
framework for functional EVT and provide basic
definitions and tools.
In particular we give a characterization of the distribution of max-stable processes via a norm on $E[0,1]$, the space of bounded functions on $[0,1]$ which have finitely many discontinuities. This norm is called $D$-norm.
In Section \ref{sec:FDA} we introduce a functional
domain of attraction approach for stochastic processes which is
more general than the usual one based on weak convergence but the introduced type of convergence is more restrictive than convergence of the finite dimensional distributions and more restrictive than hypoconvergence (see, e.g., Molchanov \cite{mol05}). The results of the foregoing sections are applied in Subsection \ref{sec:domain_of_attraction_for_copula_processes} to derive characterizations of the domain of
attraction condition for copula processes. The idea of a functional GPD is introduced in Section \ref{sec:functional_GPD}
and, finally, well-known results of the multivariate case are carried over, particularly $\delta-$neighborhoods of a functional
 GPD are considered in Subsection \ref{sec:Spectral_delta_Neighborhood_of_a_standard_GPP}.

To improve the readability of this paper we use bold face such as $\bfxi$, $\bfY$ for
stochastic processes and default font $f$, $a_n$ etc. for nonstochastic functions.
Operations on functions such as $\bfxi<f$ or $(\bfxi-b_n)/a_n$ are meant pointwise. The
usual abbreviations \textit{df, fidis, iid} and \textit{rv} for the terms \textit{distribution
function, finite dimensional distributions, independent and identically distributed} and
\textit{random variable}, respectively,  are used.

\section{Max-Stable Processes in $C[0,1]$}\label{sec:extreme_value_processes}

A \emph{max-stable process} (MSP) $\bfzeta=\left(\zeta_t\right)_{t\in[0,1]}$  realizing in $C[0,1]:=\{f:[0,1]\to\R:\ f \textrm{ continuous}\}$, equipped with the sup-norm $\norm f_\infty=\sup_{t\in[0,1]}\abs{f(t)}$, is a stochastic process with the characteristic property that its distribution is max-stable, i.e., $\bfzeta$ has the same distribution as $\max_{1\leq i\leq n}(\bfzeta_i-b_n)/a_n$ for independent copies $\bfzeta_1,\bfzeta_2,\dots$ of $\bfzeta$ and some $a_n,b_n\in C[0,1],\,a_n>0$, $n\in\N$ (c.f. de Haan and Ferreira  \cite{dehaf06}), i.e.,
\begin{equation}
\bfzeta =_D \max_{1\leq i\leq n}(\bfzeta_i-b_n)/a_n,
\label{DefinitionMaxStableProcess}
\end{equation}
the maxima being taken componentwise. In particular, $\zeta(t)$ is a max-stable real valued rv for every $t\in[0,1]$, i.e. its distribution has for some $a(t)>0$ and $b(t), \gamma(t)\in\R$ a von Mises representation (cf. Falk et al. \cite{fahure10}, de Haan and Ferreira \cite{dehaf06})
\begin{equation}
P\left(\frac{\bfzeta(t)-b(t)}{a(t)}\leq x\right)=: F_{\gamma(t)}(x)=\exp\left(-(1+\gamma(t) x)^{-1/\gamma(t)}\right),\quad \gamma x\geq -1,
\label{eq:vonMisesRep}
\end{equation}
$\gamma\in\R$, where
\begin{equation*}
F_{\gamma(t)}(x)=\left\{ \begin{array}{ll}
    0 & \text{for } \gamma(t)>0 \text{ and } x\leq -1/\gamma(t),\\
    1 & \text{for } \gamma(t)<0 \text{ and } x\geq -1/\gamma(t),\\
\exp\left(-\exp(-x)\right) &\text{for } \gamma(t)=0 \text{ and } x\in\R.
\end{array}\right.
\end{equation*}

 It was shown in Gin\'{e} et al. \cite{ginhv90} that for a continuous max-stable process those norming constants $a(t),b(t), \gamma(t)$ have to be continuous in $t\in[0,1]$.

 In the finite-dimensional case, the characterization of the max-stable distributions is typically done by characterizing some standard case (with certain margin restrictions) and reaching all other cases by (margin) transformation. Due to the results in Gin\'{e} et al. \cite{ginhv90} this is also possible for max-stable processes, so we can proceed in an analogous way.

\subsection{Standard Max-Stable Processes and the $D$-Norm in Function Spaces}\label{sec:standardMSPandD_Norm}

We call a process $\bfeta$ which realizes in $C[0,1]$ a
\emph{standard} MSP, if it is an MSP with standard negative
exponential (one-dimensional) margins, $P(\eta_t\le x)=\exp(x)$,
$x\le 0$, $t\in[0,1]$.

According to Gin\'{e} et al. \cite{ginhv90} and de Haan and Ferreira \cite{dehaf06}, a process $\bfxi$ in $C[0,1]$ is called a \emph{simple} MSP, if it is an MSP with standard Fr\'{e}chet (one-dimensional)
margins, $P(\xi_t\le x)=\exp(-1/x)$, $x> 0$, $t\in[0,1]$. We will see that each simple MSP $\bfxi$ can be transformed to a standard MSP $\bfeta$ by just transforming the univariate margins $\eta_t:=-1/\xi_t$, $0\le t\le 1$, and, vice versa, $\xi_t:=-1/\eta_t$. With this one-to-one correspondence one might consider the spaces of simple MSP and standard MSP as \emph{dual spaces}.

A crucial observation is the fact that neither a simple MSP $\bfxi$ nor a standard MSP $\bfeta$ attains the value $0$ (with probability one), which is the content of the following two auxiliary results, which are of interest of their own. Lemma \ref{simpleMSPnotzero} was already established by Gin\'{e} et al. \cite{ginhv90} using the theory of random sets. Furthermore, Theorem 9.4.1 in \cite{dehaf06} contains this assertion, too, proven by elementary probabilistic arguments.

\begin{lemma}\label{standardMSPnotzero}
Let $K$ be a compact subset of $[0,1]$ and let $\bfeta_K=(\eta_t)_{t\in K}$ be a max-stable process on $K$ with standard negative exponential margins, which realizes  in the space $\bar C^-(K):=\{f:K\to(-\infty,0],\,f \mbox{ is continuous}\}$ of nonpositive continuous functions on $K$. Then we have
\[
P\left(\max_{t\in K}\eta_t<0\right)=1.
\]
\end{lemma}

\begin{proof}
The crucial argument in this proof is the following fact. We have
for an arbitrary interval $[a,b]\subset[0,1]$
\begin{equation}\label{eqn:zero_one_law}
P\left(\max_{t\in[a,b]\cap K}\eta_t<0\right)\in\set{0,1}.
\end{equation}
This can be seen as follows. Define for $n\in\N$ and arbitrary
$\varepsilon>0$ the function
$f_{n,\varepsilon}(t):=(-\varepsilon/n)1_{[a,b]\cap K}(t)$, $t\in K$.
Let $\bfeta_K^{(1)},\bfeta_K^{(2)},\dots$ be independent copies of $\bfeta_K$.
From the max-stability of $\eta$ we obtain
\begin{align*}
P\left(\max_{t\in[a,b]\cap K}\eta_t\le \frac{-\varepsilon}
n\right)&=P\left(\bfeta_K\le f_{n,\varepsilon}\right)\\
&=\left(P\left(\bfeta_K\le f_{n,\varepsilon}\right)^n\right)^{1/n}\\
&=P\left(\max_{1\le i\le n}\bfeta_K^{(i)}\le
f_{n,\varepsilon}\right)^{1/n}\\
&=P\left(n\max_{1\le i\le n}\bfeta_K^{(i)}\le
nf_{n,\varepsilon}\right)^{1/n}\\
&=P\left(\bfeta_K\le n f_{n,\varepsilon}\right)^{1/n}\\
&=P\left(\max_{t\in[a,b]\cap K}\eta_t\le-\varepsilon\right)^{1/n}\\
&\to_{n\to\infty}1
\end{align*}
unless $P\left(\max_{t\in[a,b]\cap K}\eta_t\le-\varepsilon\right)=0$.
Equation \eqref{eqn:zero_one_law} now follows by the continuity from below of a probability measure:
\begin{align*}
P\left(\max_{t\in[a,b]\cap K}\eta_t<0\right)&=P\left(\bigcup_{n\in\N}\set{\max_{t\in[a,b]\cap K}\eta_t\le\frac{-\varepsilon}
n}\right)\\
&=\lim_{n\to\infty}
P\left(\max_{t\in[a,b]\cap K}\eta_t\le\frac{-\varepsilon} n\right).
\end{align*}

Equation \eqref{eqn:zero_one_law} implies
$$
1-P\left(\max_{t\in K}\eta_t<0\right)=P\left(\max_{t\in K}\eta_t=0\right)= P\left(\max_{t\in[0,1]\cap K}\eta_t=0\right)\in\set{0,1}.
$$
We show by a contradiction that this probability is actually zero. Assume that
it is $1$. We divide the interval $[0,1]$ into the two
subintervals $[0,1/2]$, $[1/2,1]$. Now we obtain from equation \eqref{eqn:zero_one_law}
that $P\big(\max_{t\in[0,1/2]\cap K}\eta_t=0\big)=1$ or
$P\big(\max_{t\in[1/2,1]\cap K}\eta_t=0\big)=1$. Suppose without loss of generality
that the first probability is $1$ (if one of the two intersections with $K$ is empty, the probability concerning the other one has to be equal to $1$). Then we divide the interval
$[0,1/2]$ into the two subintervals $[0,1/4]$, $[1/4,1/2]$ and
repeat the preceding arguments. By iterating, this generates a
sequence of nested intervals $I_n=[t_n,\tilde t_n]$ in $[0,1]$
with $P\left(\max_{t\in I_n\cap K}\eta_t=0\right)=1$, $\tilde t_n-
t_n=2^{-n}$, $n\in\N$, and $t_n\uparrow t_0$, $\tilde
t_n\downarrow t_0$ as $n\to\infty$ for some $t_0\in K$. From
the lower continuity of a probability measure we now conclude
\begin{align*}
0&=P\left(\eta_{t_0}=0\right)\\
&=P\left(\bigcap_{n\in\N}\set{\max_{t\in
I_n\cap K}\eta_t=0}\right)\\
&=\lim_{n\to\infty}P\left(\set{\max_{t\in
I_n\cap K}\eta_t=0}\right)\\
&=1,
\end{align*}
since $\eta_{t_0}$ is negative exponential distributed; but this is the desired contradiction.
  \end{proof}

A repetition of the arguments in the proof in de Haan and Ferreira \cite[p. 306]{dehaf06} yields the dual result:

\begin{lemma}\label{simpleMSPnotzero}
Let $K$ be a compact subset of $[0,1]$ and let $\bfxi_K=(\xi_t)_{t\in K}$ be a max-stable process on $K$ with standard Fr\'{e}chet margins, which realizes  in the space $\bar C^+(K):=\{f:K\to[0,\infty),\,f \mbox{ is continuous}\}$ of nonnegative continuous functions on $K$. Then we have
\[
P\left(\inf_{t\in K}\bfxi(t)>0\right)=1.
\]
\end{lemma}

The following crucial characterization of continuous max-stable processes is a consequence of Gin\'{e} et al. \cite[Proposition 3.2]{ginhv90}; we refer also to de Haan and Ferreira \cite[Theorem 9.4.1]{dehaf06}.

\begin{proposition}\label{prop:characterization_of_MSP}
Let $\bfZ$ be a stochastic process which realizes in $\bar C^+[0,1]$ with the properties
\begin{equation}\label{eqn:properties_of_generator}
\max_{t\in[0,1]}Z_t = m \in[1,\infty) \mbox{ a.s. and } E(Z_t)=1,\quad t\in[0,1].
\end{equation}
\begin{itemize}
  \item[(i)] A process $\bfxi$ in $\bar C^+[0,1]$ is a \emph{simple} MSP if there exists a  stochastic process $\bfZ$ as above such that for compact subsets $K_1,\dots,K_d$ of $[0,1]$ and
$x_1,\dots,x_d> 0$, $d\in\N$,
\begin{align}\label{eqn:finite_distribution_of_simple_MSP}
&P\left(\max_{t\in K_j}\bfxi_t\le x_j, 1\le j\le
d\right)\nonumber\\
&=\exp\left(-E\left(\max_{1\le j\le d} \left(\frac{\max_{t\in
K_j}Z_t}{x_j}\right)\right)\right).
\end{align}
	\item[(ii)]  A process $\bfeta$ in $\barC=\set{f\in C[0,1]:\,f\le 0}$ is a \emph{standard} MSP if there exists a stochastic process $\bfZ$ as above
such that for compact subsets $K_1,\dots,K_d$ of $[0,1]$ and
$x_1,\dots,x_d\le 0$, $d\in\N$,
\begin{align}\label{eqn:finite_distribution_of_MSP}
&P\left(\max_{t\in K_j}\bfeta_t\le x_j, 1\le j\le d\right)\nonumber\\
&=\exp\left(-E\left(\max_{1\le j\le d} \left(\abs{x_j}\max_{t\in
K_j}Z_t\right)\right)\right).
\end{align}\label{item:finite_distribution_of_MSP}
\end{itemize}

Conversely, every stochastic process $\bfZ$ in $\bar C^+[0,1]$ satisfying
\eqref{eqn:properties_of_generator} gives rise to a simple and to a standard MSP.
The connection is via \eqref{eqn:finite_distribution_of_simple_MSP} and \eqref{eqn:finite_distribution_of_MSP}, respectively.
We call $\bfZ$ \emph{generator} of $\bfeta$ and $\bfxi$.
\end{proposition}

\begin{proof}
Identify the finite measure $\sigma$ in Gin\'{e} et al. \cite[Proposition 3.2]{ginhv90} with $m(P\ast \tilde\bfZ)$, where $(P\ast \tilde\bfZ)$ denotes the distribution of some process $\tilde\bfZ\in \bar C_1^+:=\{f\in C[0,1]: f\geq 0,\norm{f}_\infty =1 \}$ and set $\bfZ=m\tilde\bfZ$, where $m$ is the total mass of the measure $\sigma$. Assertion \textit{(ii)} now follows by setting $\bfeta= -1/\bfxi$, which is well defined by Lemma \ref{simpleMSPnotzero}.
  \end{proof}

According to de Haan and Ferreira \cite[Corollary 9.4.5]{dehaf06},  condition \eqref{eqn:properties_of_generator} can be weakened to the condition $E\left(\max_{t\in[0,1]} Z_t\right)<\infty$, together with $E(Z_t)=1$, $t\in[0,1]$. While a generator $\bfZ$ is in general not uniquely determined,  the number $m=E\left(\max_{t\in[0,1]}Z_t\right)$ is, see below. We, therefore, call $m$ the \emph{generator constant} of $\bfeta$.

The above characterization implies in particular that the
fidis $G_{t_1,\dots,t_d}(\bfx)=P(\eta_{t_1}\le x_1,\dots,\eta_{t_d}\le x_d)$, $\bfx=(x_1,\dots,x_d)\in\R^d$, of $\bfeta$ are multivariate EVD with standard negative exponential
margins: We have for $0\le t_1<t_2\dots<t_d\le 1$
\begin{equation}\label{eqn:definition_of_D-norm}
-\log(G_{t_1,\dots,t_d}(\bfx))=E\left(\max_{1\le i\le
d}(\abs{x_i}Z_{t_i})\right)
=:\norm{\bfx}_{D_{t_1,\dots,t_d}},\  \bfx\le\bfzero\in\R^d,
\end{equation}
where $\norm{\cdot}_{D_{t_1,\dots,t_d}}$ is a \textit{$D$-norm} on
$\R^d$ (see Falk et al. \cite{fahure10}).

Denote by $E[0,1]$ the set of all functions on $[0,1]$ that are
bounded and which have only a finite number of discontinuities. Furthermore, denote by $\bar E^-[0,1]$ those functions in $E[0,1]$ which do not attain positive values.

\begin{definition} For a generator process $\bfZ$ in $\bar C^+[0,1]$ with properties \eqref{eqn:properties_of_generator} set
\[
\norm{f}_D:=E\left(\sup_{t\in [0,1]}\left(\abs{f(t)}Z_t\right)\right),\qquad f\in E[0,1].
\]
Then $\norm{\cdot}_D$ obviously defines a norm on $E[0,1]$, called a
\emph{$D$-norm with generator $\bfZ$}.
\end{definition}

The sup-norm $\norm f_\infty:=\sup_{t\in[0,1]}\abs{f(t)}$, $f\in E[0,1]$,
is a particular $D$-norm with constant generator $Z_t=1$,
$t\in[0,1]$. It is, moreover, the least $D$-norm, as
\begin{equation}\label{sup_f_D_sandwich}
\norm f_\infty\le \norm f_D\le m \norm f_\infty,\qquad f\in
E[0,1],
\end{equation}
for any $D$-norm $\norm \cdot_D$ whose generator satisfies
$E\left(\max_{t\in[0,1]}Z_t\right)=m$. For the constant function $f=1$ we obtain
$\norm 1_D=m$.

Note that inequality \eqref{sup_f_D_sandwich} implies that each functional $D$-norm is equivalent to the sup-norm. This, in turn, yields that no $L_p$-norm $\norm f_p=(\int_0^1 \abs{f(t)}^p\,dt)^{1/p}$ with $p\in(0,\infty)$ is a functional $D$-norm.

The following result provides the distribution function of a standard max-stable process in terms of the functional $D$-norm, which is in high conformity with the finite dimensional case (cf. Falk et al. \cite[Section 4.4]{fahure10}).

\begin{lemma}\label{prop:distribution_function_of_standard_MSP}
Let $\bfeta$ be a standard MSP with generator $\bfZ$. Then we have for each $ f\in \bar E^-[0,1]$
\begin{equation}\label{eq:distribution_function_of_standard_MSP}
P(\bfeta\le f) = \exp\left(-\norm f_D\right)= \exp\left(-E\left(\sup_{t\in [0,1]}\left(\abs{f(t)}Z_t\right)\right)\right).
\end{equation}

Conversely, if there is some $\bfZ$ with properties \eqref{eqn:properties_of_generator} and some $\bfeta\in C^-[0,1]$ which satisfies \eqref{eq:distribution_function_of_standard_MSP}, then $\bfeta$ is standard max-stable with generator $\bfZ$.
\end{lemma}

\begin{proof}
Let $Q=\set{q_1,q_2\dots}$ be a denumerable and dense subset of
$[0,1]$ which contains those finitely many points at which $f\in \bar E^-[0,1]$ has a discontinuity.

We obtain from the continuity of $\bfeta$, the continuity from above of each probability measure and equation
\eqref{eqn:definition_of_D-norm}
\allowdisplaybreaks[4]\begin{align*}
P(\bfeta \le f) &=P\left(\bigcap_{d\in\N}\set{\eta_{q_j}\le f(q_j),\,1\le j\le d}\right)\\
&=\lim_{d\to\infty} P\left(\eta_{q_j}\le f(q_j),\,1\le j\le d\right)\\
&=\lim_{d\to\infty} \exp\left(-E\left(\max_{1\le j\le d}\left(\abs{f(q_j)}Z_{q_j}\right)\right)\right)\\
&=\exp\left(- E\left(\lim_{d\to\infty} \max_{1\le j\le d}\left(\abs{f(q_j)}Z_{q_j}\right)\right)\right)\\
&=\exp\left(- E\left(\sup_{t\in[0,1]}\left(\abs{f(t)}Z_t\right)\right)\right)\\
&=\exp\left(-\norm f_D\right),
\end{align*}
where the third to last equation follows from the dominated
convergence theorem.

If some $\bfZ$ has properties \eqref{eqn:properties_of_generator} it gives rise to some standard max-stable process $\widehat \bfeta$ due to Proposition \ref{prop:characterization_of_MSP}. But the fidis of $\widehat \bfeta$ given by \eqref{eqn:finite_distribution_of_MSP} and those of $\bfeta$ given by \eqref{eq:distribution_function_of_standard_MSP} coincide, so $\widehat \bfeta =_D \bfeta$ follows.
  \end{proof}

The extension to $f\in \bar E^-[0,1]$ allows the incorporation of the fidis of $\bfeta$ into the preceding representation: Choose indices $0\le t_1<\dots<t_d\le 1$ and
numbers $x_i\le 0$, $1\le i\le d$, $d\in\N$. The function
\[
f(t)=\sum_{i=1}^d x_i 1_{\set{t_i}}(t)
\]
is an element of $\bar E^-[0,1]$ with the property
\begin{align*}
P(\bfeta\le f)&= \exp\left(-\norm f_D\right)\\
&=\exp\left(-E\left(\max_{1\le i\le d}\left(\abs{x_i} Z_{t_i}\right)\right)\right)\\
&=\exp\left(-\norm{\bfx}_{D_{t_1,\dots,t_d}}\right).
\end{align*}

This is one of the reasons, why we prefer standard MSPs (with standard negative exponential margins),
whereas de Haan and Ferreira [9], for instance, consider simple MSPs (with standard Fr\'{e}chet margins).

We can now, for example, extend Takahashi's \cite{taka88} characterization of the
maximum-norm in $\R^d$ to the functional space $E[0,1]$.

\begin{lemma}[Functional  Takahashi] Let $\norm \cdot_D$ be an arbitrary $D$-norm on
$E[0,1]$ with generator $\bfZ$. Then
\begin{align*}
&\norm f_D=\norm f_\infty\mbox{ for at least one }f\in
E[0,1]\mbox{
with }f(t)\not=0,\,t\in[0,1]\\
&\iff\norm\cdot_D=\norm\cdot_\infty\mbox{ on }E[0,1].
\end{align*}
\end{lemma}

\begin{proof}
Let $f\in E[0,1]$ have the property $\norm f_D=\norm f_\infty$. Suppose first that $\norm f_\infty$ is attained on $[0,1]$, i.e., there exists $t_0\in[0,1]$ such that $\abs{f(t_0)}=\sup_{t\in[0,1]}\abs{f(t)}$. Then we obtain for
arbitrary indices $0\le t_1<\dots<t_d\le 1$
\[
\abs{f(t_0)}Z_{t_0}\le
\max_{i=0,1,\dots,d}\left(\abs{f(t_i)}Z_{t_i}\right)\le
\sup_{t\in[0,1]}\left(\abs{f(t)}Z_t\right)
\]
and, thus,
\begin{align*}
\norm f_\infty&= E\left(\abs{f(t_0)}Z_{t_0}\right)\\
&\le
E\left(\max_{i=0,1,\dots,d}\left(\abs{f(t_i)}Z_{t_i}\right)\right)\\
&=\norm{\left(f(t_0),\dots, f(t_d)\right)}_{D_{t_0,\dots,t_d}}\\
&\le E\left(\sup_{t\in[0,1]}\left(\abs{f(t)}Z_t\right)\right)\\
&=\norm f_D\\
&=\norm f_\infty,
\end{align*}
i.e.,
\[
\norm{\left(f(t_0),\dots,
f(t_d)\right)}_{D_{t_0,\dots,t_d}}=\norm{\left(f(t_0),\dots,
f(t_d)\right)}_\infty.
\]
Takahashi's Theorem \cite{taka88} for the finite-dimensional
Euclidean space now implies
\begin{equation}\label{eqn:functional_takahashi}
\norm\bfx_{D_{t_0,\dots,t_d}}=\norm\bfx_\infty,\qquad
\bfx\in\R^{d+1},
\end{equation}
for arbitrary $0\le t_1<\dots<t_d\le 1$, $d\in\N$. This, in turn
implies that $Z_t=Z_{t_0}$, $t\in[0,1]$, a.s., which can be seen
as follows. Choose $f$ the constant function $1$ and let $\mathbb
Q\cap [0,1]=\set{t_1,t_2,\dots}$. Then we have by equation
\eqref{eqn:functional_takahashi} for arbitrary $s\in
\set{t_1,t_2,\dots}$ if $d$ is large
\[
1=E(Z_s)\le E\left(\max_{i=0,\dots,d} Z_{t_i}\right)=1
\]
and, thus, by the dominated convergence theorem and the continuity
of $(Z_t)_{t\in[0,1]}$
\[
1=E(Z_s)\le E\left(\sup_{t\in[0,1]} Z_t\right)=1.
\]
But this implies $0=E\left(\sup_{t\in[0,1]} Z_t-Z_s\right)$ and,
hence, $Z_s=\sup_{t\in[0,1]} Z_t$ a.s., which yields
$Z_t=Z_{t_0}$, $t\in[0,1]$, a.s. by the continuity of the process
$\bfZ$. This implies $\norm f_D=\norm f_\infty$, $f\in E[0,1]$.

Suppose next that $\norm f_\infty$ is not attained. Then there
exists a sequence of indices $t_n$, $n\in\N$, in $[0,1]$ with
$t_n\to_{n\to\infty}t_0\in[0,1]$ and
$\lim_{n\to\infty}\abs{f(t_n)}=\norm f_\infty$. From the
continuity of the process $\bfZ$ we obtain
\[
\norm f_\infty Z_{t_0}=\lim_{n\to\infty}\abs{f(t_n)}Z_{t_n}
\]
and, thus,
\[
\norm f_\infty Z_{t_0}\le
\sup_{t\in[0,1]}\left(\abs{f(t)}Z_t\right).
\]
Choose arbitrary indices $0\le t_1<\dots,t_d\le 1$. Then
\begin{align*}
\norm f_\infty&=E\left(\norm f_\infty Z_{t_0}\right)\\
&\le E\left(\max\set{\norm f_\infty Z_{t_0},
\abs{f(t_1)}Z_{t_1},\dots, \abs{f(t_d)}Z_{t_d}}\right)\\
&\le E\left(\sup_{t\in[0,1]}\left(\abs{f(t)}Z_t\right)\right)\\
&=\norm f_\infty.
\end{align*}
From Takahashi's Theorem \cite{taka88} we now deduce that
\[
E\left(\max_{i=0,\dots,n}\left(\abs{f(t_i)}Z_{t_i}\right)\right)=\norm
f_\infty
\]
for arbitrary $f\in E[0,1]$. Concluding as above yields the
assertion.
  \end{proof}

The characterization of the distribution of a standard EVP $\bfeta$ via the $D$-norm has some further implications. The following assertions are essentially due to property \eqref{sup_f_D_sandwich}.

\begin{lemma}\label{dfofEVPcontinuous}
 Let $\bfeta$ in $\barC$ be a standard MSP and consider its  distribution function
 $$
 G(f)= P\left(\bfeta\leq f\right)=\exp\left(-\norm f_D\right),\quad f\in\barE.
 $$
 Then we have
 \begin{description}
	\item[(i)] $G(\cdot)$ is continuous with respect to the sup-norm.
\item[(ii)] For every $f\in\barE$ we have
\begin{equation*}
P\left(\bfeta\leq f\right)\ =\ P\left(\bfeta < f\right);
\end{equation*}
in particular, the sets $\left\{ g\in\barC:g(t)\leq f(t),\ \text{for all } t\in[0,1]\right\}$ are \emph{continuity sets} with respect to the distribution on $(\barC,\norm\cdot_\infty)$ of the standard MSP $\bfeta$.
\end{description}
\end{lemma}

The next auxiliary result provides some properties of the  survivor function  of a standard MSP.

\begin{lemma}\label{lem:expansion_of_survivor_function_of_eta}
Let $\bfeta$ be a standard MSP with generator $\bfZ$. Then we
obtain for $f\in\bar E^-[0,1]$:
\begin{itemize}\itemsep4mm
\item[(i)]
$\displaystyle
P(\bfeta>f) \ge 1-\exp\left(-E\left(\inf_{0\le t\le
1}(\abs{f(t)}Z_t)\right)\right);
$
\item[(ii)]
$\displaystyle
\lim_{s\downarrow 0}\frac{P(\bfeta > sf)}s=E\left(\inf_{0\le t\le
1}(\abs{f(t)}Z_t)\right).
$

\end{itemize}
\end{lemma}

Note that it is easy to find examples of standard MSP $\bfeta$ and $f\in\bar E^-[0,1]$ with a strict inequality in part (i) of the preceding lemma.

\begin{proof}
Due to the continuity of $\bfeta$ and $\bfZ$ it is sufficient to consider $f\in\bar E^-[0,1]$ with $\sup_{0\le t\le 1}f(t)<0$. From Subsection \ref{subsec:arbitrary_margins} below we know that
\[
\xi_t:=-\frac 1{\eta_t},\qquad 0\le t\le 1,
\]
defines a continuous max-stable process $\bfxi=(\xi_t)_{0\le t\le
1}$ on $[0,1]$ with standard Fr\'{e}chet margins and Proposition 3.2
in Gin\'{e} et. al \cite{ginhv90} yields
\[
\bfxi=_D \max_i \bfY_i
\]
in $\bar C^+[0,1]$, where $\bfY_1,\bfY_2,\dots$ are the points
(functions in $\bar C^+[0,1]$) of a Poisson process $N$ with
intensity measure $\nu$ given by $d\nu=d\sigma\times dr/r^2$ on
$\bar C_1^+[0,1]\times (0,\infty)=:C[0,1]^+=\set{h\in C[0,1]:\,h\ge
0,\,h\not=0}$. By $\bar C_1^+[0,1]$ we denote the space of those
functions $h$ in $\bar C^+[0,1]$ with $\norm h_\infty=\sup_{0\le
t\le 1}\abs{h(t)}=1$. The (finite) measure $\sigma$ is given by
$\sigma(\cdot)=mP(\tilde \bfZ\in\cdot)$, where
$\tilde\bfZ:=\bfZ/m$ and $m$ is the generator constant pertaining
to $\bfZ$. Note that $m$ coincides with the total mass of
$\sigma$.\\
Observe $P\left(\eta_{t}> f(t), \text{ for all }t\in[0,1]\right)=1-P\left(\eta_{t}\leq f(t), \text{ for some }t\in[0,1]\right)$, and we obtain
\begin{align*}
&P\left(\eta_{t}\leq f(t), \text{ for some }t\in[0,1]\right)\\
&=P\left(\xi_{t}\le \frac 1{\abs{f(t)}}, \text{ for some }t\in[0,1]\right)\\
&=P\left(\text{for some }t\in[0,1],\, \forall i\in\N: \bfY_i(t)\le \frac 1{\abs{f(t)}},\, \forall i\in\N\right)\\
&\le P\left(\forall i\in\N,\,\text{for some }t\in[0,1]: \bfY_i(t)\le \frac 1{\abs{f(t)}}\right)\\
&=P\left(N\left(\set{g\in C[0,1]^+:\,g(t)> \frac 1{\abs{f(t)}},\,t\in[0,1]}\right)=0\right)\\
&=\exp\left(-\nu\left(\set{g\in C[0,1]^+:\,g(t)\abs{f(t)}>1,\,t\in[0,1]}\right)\right)\\
&=\exp\left(-\nu\left(\set{(h,r)\in\bar C_1^+[0,1]\times (0,\infty):\, r h(t)\abs{f(t)} > 1,\, t\in[0,1]}\right)\right)\\
&=\exp\left(-\int_{\set{(h,r)\in\bar C_1^+[0,1]\times (0,\infty):\, r h(t)\abs{f(t)} > 1,\, t\in[0,1]}} \frac 1{r^2}\,dr\,\sigma(dh)\right)\\
&=\exp\left(-\int_{\bar C_1^+[0,1]}\int_{1/\inf_{t\in[0,1]}(h(t)\abs{f(t)})}^\infty  \frac 1{r^2}\,dr\,\sigma(dh)\right)\\
&=\exp\left(-\int_{\bar C_1^+[0,1]} \inf_{t\in[0,1]}(h(t)\abs{f(t)}) \,\sigma(dh)\right)\\
&=\exp\left(-E\left(\inf_{t\in[0,1]}(\abs{f(t)}Z_t)\right)\right).
\end{align*}
which is assertion $(i)$. Next we establish the inequality
\begin{equation}\label{eqn:P(eta>s)<E(min_Z)}
\limsup_{s\downarrow 0}\frac{P(\bfeta > sf)}s\le E\left(\min_{1\le j\le m}\left(\abs{f(t_j)}Z_{t_j}\right)\right),
\end{equation}
where $\{t_1,t_2,\ldots\}$ is a denumerable dense subset of [0,1], which contains those finitely many points $t_i$ at which the function $f$ is discontinouous. Not that in this case $\lim_{n\to\infty}E\left(\min_{1\le j\le m}\left(\abs{f(t_j)}Z_{t_j}\right)\right)= E\left(\min_{t\in[0,1]}\left(\abs{f(t)}Z_{t}\right)\right)$ because of the dominated convergence theorem.

The inclusion-exclusion theorem implies
\begin{align*}
&P(\bfeta>sf)\\
&\le P\left(\bigcap_{j=1}^m\set{\eta_{t_j}>sf(t_j)}\right)\\
&=1- P\left(\bigcup_{j=1}^m\set{\eta_{t_j}\le sf(t_j)}\right)\\
&=1-\sum_{\emptyset\not= T\subset\set{1,\dots,m}} (-1)^{\abs T - 1} P\left(\bigcap_{j\in T}\set{\eta_{t_j}\le sf(t_j)}\right)\\
&=1-\sum_{\emptyset\not= T\subset\set{1,\dots,m}} (-1)^{\abs T - 1} \exp\left(- s E\left( \max_{j\in T}\left(\abs{f(t_j)}Z_{t_j}\right)\right)\right)\\
&=:1-H(s)\\
&=H(0)-H(s),
\end{align*}
where the function $H$ is differentiable and, thus,
\begin{align*}
\limsup_{s\downarrow 0}\frac{P(\bfeta > sf)}s&\le -\lim_{s\downarrow 0}\frac{H(s)-H(0)}s\\
&=-H'(0)\\
&= \sum_{\emptyset\not= T\subset\set{1,\dots,m}} (-1)^{\abs T - 1}   E\left( \max_{j\in T}\left(\abs{f(t_j)}Z_{t_j}\right)\right)\\
&= E\left( \min_{j\in T}\left(\abs{f(t_j)}Z_{t_j}\right)\right),
\end{align*}
since $\sum_{\emptyset\not= T\subset\set{1,\dots,m}} (-1)^{\abs T - 1} \max_{j\in T}a_j=\min_{1\le j\le m}a_j$ for arbitrary numbers $a_1,\dots,a_m\in\R$, which can be seen by induction. This implies equation \eqref{eqn:P(eta>s)<E(min_Z)}. Part $(ii)$ is now a straightforward consequence of $(i)$ and \eqref{eqn:P(eta>s)<E(min_Z)}.
 
\end{proof}

\subsection{Transformation to Arbitrary Margins}\label{subsec:arbitrary_margins}

Next we recall that the characterization in Proposition \ref{prop:characterization_of_MSP} is
sufficient to cover all max-stable processes in $C[0,1]$.

Let $\bfzeta$ be an arbitrary max-stable process in $C[0,1]$ and $a>0$, $b$, $\gamma$
the continuous functions, for which the von Mises representation
\begin{equation*}
P\left(\frac{\bfzeta(t)-b(t)}{a(t)}\leq
x\right)=\exp\left(-(1+\gamma(t) x)^{-1/\gamma(t)}\right),\qquad t\in[0,1],
\end{equation*}
holds, cf. equation \eqref{eq:vonMisesRep}.

We deduce from Gin\'{e} et al. \cite{ginhv90} that the transformation
\begin{equation}\label{eq:transformationtostandardcase}
\bfeta(t):= \left\{ \begin{array}{ll}
    -\left(1+ \frac{\gamma(t)}{a(t)}(\bfzeta(t)-b(t))  \right)^{-1/\gamma(t)} & \textrm{for } \gamma(t)\neq0 \\[2mm]
    -\exp\left(-(\bfzeta(t)-b(t))/a(t)\right)& \textrm{for } \gamma(t)=0.
\end{array}\right.
\end{equation}
is well-defined and continuous, and elementary transformations yield that $\bfeta$ is a standard MSP.

By inverting equation \eqref{eq:transformationtostandardcase} we get 
\begin{equation*}
\bfzeta(t)= \left\{ \begin{array}{ll}
    \frac{-a(t)}{\gamma(t)} \left(1- \left(-\bfeta(t)\right)^{-\gamma(t)}\right)+b(t)   & \textrm{ for } \gamma(t)\neq0 \\[2mm]
    -a(t)\log\left(-\bfeta(t)\right)+b(t)& \textrm{ for } \gamma(t)=0.
\end{array}\right.
\end{equation*}
for $t\in[0,1]$.

Thus, the functional df of an arbitrary max-stable process $\bfzeta$ in $C[0,1]$ can be written by means of the $D$-norm: we get for  $f\in E[0,1]$ 
\begin{eqnarray*}
P(\bfzeta \leq f) &=& P\left(\bfeta\leq \PsiO(f)\right)\\[3mm]
   &=& \exp\left(-\norm{\PsiO(f)}_D\right), 
\end{eqnarray*}
where we define for functions $f\in E[0,1]$
\begin{equation*}
\PsiO(f):= \PsiO(f(t)):=\left\{ \begin{array}{ll}
    -\left(1+ \frac{\gamma(t)}{a(t)}(f(t)-b(t))  \right)^{-1/\gamma(t)} & \textrm{for } \gamma(t)\neq0 \\[2mm]
    -\exp\left(-(f(t)-b(t))/a(t)\right)& \textrm{for } \gamma(t)=0.
\end{array}\right.
\end{equation*}

\section{Functional Domain of Attraction}\label{sec:FDA}
\subsection{Functional Domain of Attraction of a standard MSP}

We say that a stochastic process $\bfY$ in $C[0,1]$ is \emph{in the
functional domain of attraction} of a standard MSP   $\bfeta$,
denoted by $\bfY\in \mathcal D(\bfeta)$, if there are functions
$a_n\in C^+[0,1]:=\set{f\in C[0,1]:f>0}$, $b_n\in C[0,1]$, $n\in\N$, such that
\begin{equation}\label{cond:definition_of_domain_of_attraction}
\lim_{n\to\infty}P\left(\frac{\bfY-b_n}{a_n}\le f\right)^n =
P(\bfeta\le f)=\exp\left(-\norm f_D\right)
\end{equation}
for any $f\in\barE$. Note that this condition is equivalent with
\begin{equation}\label{cond:_equivalent_definition_of_domain_of_attraction}
 \lim_{n\to\infty} P\left(\max_{1\le i\le n}\frac{\bfY_i-b_n}{a_n}\le f\right)=P(\bfeta\le f) \tag{\ref{cond:definition_of_domain_of_attraction}'}
\end{equation}
for any $f\in\barE$, where $\bfY_1,\bfY_2,\dots$ are independent
copies of $\bfY$.\par
\bigskip
Due to the continuity of the functional df of $\bfeta$, we get immediately the following assertion.

\begin{lemma}\label{open_convergence}
 There is $\bfY\in \mathcal D(\bfeta)$ for some standard EVP $\bfeta$, i.e. \eqref{cond:_equivalent_definition_of_domain_of_attraction} holds, if, and only if
\begin{equation}
\lim_{n\to\infty}P\left( \max_{1\leq i\leq n}(\bfY_i-b_n)/a_n < f\right) = P\left( \bfeta < f\right)
\label{eq:open_convergence}
\end{equation} for every $f\in\barE$, with $\bfY_i,\,a_n,\,b_n$ as before.
\end{lemma}

\begin{proof}
Set $\bfX_n:=\max_{1\leq i\leq n}(\bfY_i-b_n)/a_n$. If \eqref{cond:_equivalent_definition_of_domain_of_attraction} holds, we get the inequality
\begin{equation*}
\limsup_{n\to \infty}P(\bfX_n<f) \leq \lim_{n\to\infty}P(\bfX_n\leq f) = P(\bfeta\leq f) = P\left( \bfeta \leq f\right)
\end{equation*}
for every $f\in\barE$, see Lemma \ref{dfofEVPcontinuous}.

On the other hand, for all $f\in\barE$ and every $\varepsilon>0$:
\begin{equation*}
P(\bfeta\leq f-\varepsilon)= \lim_{n\to\infty}P(\bfX_n\leq f-\varepsilon) \leq \liminf_{n\to \infty}P(\bfX_n<f).
\end{equation*}
As $G(f)=P(\bfeta\leq f)$ is continuous in $f$ with respect to the $\sup-$norm, cf. Lemma \ref{dfofEVPcontinuous}, \eqref{eq:open_convergence} follows. The reverse implication follows with analogous arguments.
 \end{proof}

There should be no risk of confusion with the notation of domain
of attraction in the sense of weak convergence of stochastic
processes as investigated in de Haan and Lin \cite{dehal01}. But to
distinguish between these two approaches we will consistently
speak of \emph{functional} domain of attraction in this paper,
if the above definition is meant. Actually, this definition of
domain of attraction is less restrictive as the next lemma shows.

\begin{proposition}\label{weakimpliesfunctional}
Suppose that $\bfY$ in $\bar C^-[0,1]$ and let $\bfY_1,\bfY_2,\dots$ be independent copies of $\bfY$. If the sequence  $\bfX_n=\max_{1\le i\le
n}\left(\left(\bfY_i-b_n)/a_n\right)\right)$ of continuous processes converges weakly in
$\bar C^-[0,1]$, equipped with the sup-norm $\norm\cdot_\infty$,
to the standard MSP $\bfeta$, then $\bfY\in\mathcal D(\bfeta)$.
\end{proposition}

\begin{proof}
The Portmanteau Theorem (see, e.g., Billingsley \cite{billi99}) characterizes weak convergence in particular in terms of convergence of the masses of all continuity sets. So Lemma \ref{dfofEVPcontinuous} immediately implies the assertion.
  \end{proof}

Examples of  continuous processes in $\bar C^-[0,1]$, whose properly normed maxima of iid copies converge weakly to an MSP and which obviously satisfy condition \eqref{cond:definition_of_domain_of_attraction}, are the \textit{GPD-processes} introduced by Buishand et al. \cite{buihz08}. We consider these generalized Pareto processes in Section \ref{sec:functional_GPD} below.

The following example shows, that convergence of a sequence of functional df of some continuous processes $\bfeta_n$ to the functional df of some standard MSP $\bfeta$ does in general not imply weak convergence in $C[0,1]$.

\begin{example}
Let $\bfeta$ be a standard MSP with generator $\bfZ=(Z_t)_{t\in[0,1]}$ satisfying $E\left(\min_{t\in[0,1]} Z_t\right)>0$; note that this is equivalent to $P\left(\min_{t\in[0,1]} Z_t>0\right)>0$.

Let $U$ be a uniformly on $(0,1)$ distributed rv, which is independent of $\bfeta$. Define for $u\in[0,1]$  the triangle shaped continuous function  $\Delta_n^u:[0,1]\to[0,1]$ by
\[
\Delta_n^u(t):=\begin{cases}
1,&\mbox{if }t=u,\\
0,&\mbox{if }t\not\in[u-2^{-n},u+2^{-n}],\\
\mbox{linearly interpolated elsewhere}.&
\end{cases}
\]

Set
\[
\bm{\eta}_n:=\bfeta-\Delta_n^U,\qquad n\in\N.
\]
Note that $\bm{\eta}_n\le\bfeta$. We get on the one hand

\begin{align*}
P(\bfeta\le f)&\le P(\bm{\eta}_n\le f)\\
&=P\left(\bfeta\le f+ \Delta_n^U\right)\\
&=\int_0^1P\left(\bfeta\le f+ \Delta_n^u\right)\,du\\
&=\int_0^1P\left(\eta(t)\le f(t)+\Delta_n^u(t),t\in[0,1]\right)\,du\\
&\le \int_0^1 P(\eta(t)\le f(t),t\not\in[u-2^{-n},u+2^{-n}])\,du\\
&=\int_0^1 \exp\left(-E\left(\sup_{t\not\in[u-2^{-n},u+2^{-n}]}(\abs{f(t)}Z_t)\right)\right)\,du\\
&\to_{n\to\infty} \int_0^1 \exp\left(-E\left(\sup_{t\in[0,1]}(\abs{f(t)}Z_t)\right)\right)\,du\\
&=P(\bfeta\le f)
\end{align*}
by the continuity of $\bfZ$, the continuity up to finitely many points of $f$ and the dominated convergence theorem.

On the other hand, $\bfeta_n$ does not converge weakly to $\bfeta$ in $C[0,1]$: If that would be the case, by the Portmanteau theorem,
\[
\liminf_{n\to\infty}P(\bm{\bfeta}_n\in \mathcal O)\ge P(\bfeta\in \mathcal O)
\]
should hold for every open subset $\mathcal O$ of $C[0,1]$ (with respect to the maximum distance $\norm{f-g}_\infty=\max_{t\in[0,1]}\abs{f(t)-g(t)}$).

Choose a constant $c<-1$. Then the set $\set{g\in C[0,1]:\,g>c}$ is an open subset of $C[0,1]$ and, hence, we should have
\begin{equation}\label{eqn:liminf_>_P(beta>c)}
\liminf_{n\to\infty}P(\bm{\bfeta}_n>c)\ge P(\bfeta>c).
\end{equation}

We know from Lemma \ref{lem:expansion_of_survivor_function_of_eta} that
\[
P(\bfeta>c)\ge 1-\exp\left(-\abs c E\left(\min_{t\in[0,1]}Z_t\right)\right),
\]
and we get
\begin{align*}
P(\bm{\eta}_n>c)&=P(\bfeta-\Delta_n^U>c)\\
&=\int_0^1 P(\bfeta>c+\Delta_n^u)\,du\\
&\le \int_0^1 P(\eta(u)>c+1)\,du\\
&=\int_0^1 1-\exp(c+1)\,du\\
&=1-\exp(c+1)\\
&<1-\exp\left(cE\left(\min_{t\in[0,1]}Z_t\right)\right)\\
&\le P(\bfeta>c),
\end{align*}
provided the constant $c$ satisfies in addition $c\left(1-E\left(\min_{t\in[0,1]}Z_t\right)\right)>-1$. Note that $E\left(\min_{t\in[0,1]}Z_t\right)\le 1$ anyway. But this contradicts equation \eqref{eqn:liminf_>_P(beta>c)}.
\end{example}

By now, we have shown that functional domain of attraction is less restrictive than the domain of attraction in the sense of weak convergence. In turn, functional domain of attraction obviously implies  convergence of the fidis, and, moreover, hypoconvergence of the normed maximum-process to the standard MSP in the sense of Molchanov \cite{mol05} is implied, which can be seen as follows.

The following result is a reformulation of Proposition 3.15 in Molchanov \cite{mol05} for continuous processes and hypoconvergence; note that every continuous process is a so-called normal integrand in the sense of \cite[Definition 3.5]{mol05}.
 \begin{proposition}\label{prop:hypoconvergence}
 A sequence of continuous processes $(\bfzeta_n)_{n\in\N}$ in $C[0,1]$ weakly hypoconverges to  $\bfzeta$ in $C[0,1]$ if, and only if,
  \begin{equation*}
P\left( \sup_{x\in K_i}\bfzeta_n(x)<t_i, i=1,\ldots,m\right)\to_{ n\to\infty} P\left( \sup_{x\in K_i}\bfzeta(x)<t_i, i=1,\ldots,m\right)
\end{equation*}
for all $m\in\N, t_1,\ldots,t_m\in\R$ and $K_1,\ldots,K_m\subset[0,1]$ being finite unions of (closed) intervals satisfying the continuity condition
\begin{equation*}
P\left( \sup_{x\in \bar K_i}\bfzeta(x)< t_i\right) = P\left( \sup_{x\in K_i^\circ}\bfzeta(x)\leq t_i\right),\quad i=1,\ldots,m.
\end{equation*}
Here $\bar K$ and $K^\circ$ denotes the closure and the interior of an interval $K\subset[0,1]$ with respect to the standard topology on $\R$, respectively.
  \end{proposition}

\bigskip
Let $\bfY\in\mathcal D(\bfeta)$ for some standard MSP $\bfeta$ in the sense of definition \eqref{cond:definition_of_domain_of_attraction}. For the sake of simplicity and without loss of generality let $K_i\subset[0,1]$ be disjoint intervals, $x_i\in (-\infty,0], i=1,\ldots,m$, and set for $t\in[0,1]$
\begin{equation*}
f(t):= \sum_{i=1}^m x_i\bfone_{K_i}(t).
\end{equation*}
Then $f\in\barE$ and \eqref{eq:open_convergence} reads, with $\bfX_n=\max_{1\le i\le
n}\left(\left(\bfY_i-b_n)/a_n\right)\right)$,
\begin{equation*}
\lim_{n\to\infty}P\left( \sup_{t\in K_i}\bfX_n(t) < x_i,i=1,\ldots, n\right) = P\left( \sup_{t\in K_i}\bfeta(t) < x_i,i=1,\ldots, n\right).
\end{equation*}

Thus, hypoconvergence of the normalized maximum process follows from Proposition \ref{prop:hypoconvergence}.

Note that hypoconvergence of the normalized maximum process does in general not imply convergence in the sense of \eqref{eq:open_convergence}, since the continuity condition in Proposition \ref{prop:hypoconvergence} excludes convergence for closed subsets $K\subset[0,1]$ of the form $K=\{t\}, t\in[0,1]$.

\subsection{Domain of Attraction for Copula Processes}\label{sec:domain_of_attraction_for_copula_processes}

Let $\bfY=(Y_t)_{t\in[0,1]}$ in $C[0,1]$ be a stochastic process with
identical continuous marginal df $F$. Set
\begin{equation}\label{eq:defcopulaprocess}
\bfU=(U_t)_{t\in[0,1]}:=(F(Y_t))_{t\in[0,1]},
\end{equation}
which is the \emph{copula process} corresponding to $\bfY$. Note that each onedimensional marginal distribution of $\bfU$ is
the uniform distribution on $[0,1]$.

Suppose that the copula process corresponding to $\bfY$ is in the
functional domain of attraction of a standard MSP $\bfeta$,
representable as in Proposition
\ref{prop:characterization_of_MSP}. Then we know from Aulbach et
al. \cite{aulbf11} that for $d\in\N$ the copula $C_d$
corresponding to the rv $(Y_{i/d})_{i=1}^d$  satisfies the
equation
\begin{equation}\label{eqn:expansion_of_copula}
C_d(\bfy)=1-\norm{\bfone-\bfy}_{D_d}+o\left(\norm{\bfone-\bfy}_{\infty}\right),
\end{equation}
as $\norm{\bfone-\bfy}_{\infty}\to\bfzero$, uniformly in
$\bfy\in[0,1]^d$, where the $D$-norm is given by
\[
\norm{\bfx}_{D_d} =E\left(\max_{1\le i\le
d}\left(\abs{x_i}Z_{i/d}\right)\right),\qquad \bfx\in\R^d.
\]
We are going to establish an analogous result for the functional domain of attraction.

Let $\bfeta$ be a standard MSP with functional df $G$, and let
$\bfY$ be an arbitrary stochastic process in $C[0,1]$. By taking
logarithms, we obtain the following equivalences with some norming
functions $a_n\in C^+[0,1]$, $b_n\in C[0,1]$, $n\in\N$:
\begin{align}
&\bfY\in \mathcal{D}(\bfeta) \mbox{ in the sense of condition \eqref{cond:definition_of_domain_of_attraction}}\nonumber\\
&\iff P\left(\frac{\bfY-b_n}{a_n}\le f\right)^n=\exp\left(-\norm f_D\right)+o(1),\; f\in\barE,\mbox{ as } n\to\infty,\nonumber\\
&\iff P\left(\frac{\bfY-b_n}{a_n}\le f\right) = 1 - \frac 1n
\norm f_D+o\left(\frac 1n\right),\; f\in\barE,\mbox{ as }
n\to\infty\nonumber.
\end{align}

Let $\bfU$ be a copula-process as defined in \eqref{eq:defcopulaprocess} and set $H_f(s):=P(\bfU-1\le
s\abs f),\;s\le 0,\;f\in\barE$. Note that $H_f(\cdot)$ defines a
univariate df on $(-\infty,0]$. The family $\mathcal
P:=\set{H_f:\,f\in\barE}$ of univariate df is the
\emph{spectral decomposition} of the df $H(f)=P(\bfU-1\le f)$,
$f\in\barE$ of $\bfU-1$. This extends the spectral decomposition of a multivariate df in Falk et al. \cite[Section 5.4]{fahure10}. Standard arguments yield the next result.
\begin{proposition}\label{prop:equivalences_to_FuDA}
The following equivalences hold:
\begin{eqnarray}
\lefteqn{\bfU\in \mathcal{D}(\bfeta) \mbox{ in the sense of condition \eqref{cond:definition_of_domain_of_attraction}}}\nonumber\\
\quad&\Leftrightarrow& P\left(\bfU-1\le \frac f n\right)=1-\norm{\frac f n}_D+o\left(\frac 1 n \right),\, f\in\barE,\mbox{as }n\to\infty, \nonumber\\
&\Leftrightarrow& H_f(s)=1+s\norm{f}_D+o(s), \,f\in\barE,\mbox{as }s\uparrow 0, \label{eqn:equivalent_formulation_of_domain_of_attraction}
\end{eqnarray}
\end{proposition}

\begin{remark}\upshape Characterization \eqref{eqn:equivalent_formulation_of_domain_of_attraction} entails in particular
 that $H_f(s)$ is  differentiable from the left in $s=0$ with derivative $h_f(0):=\frac{d}{ds}H_f(s)|_{s=0}=\norm f_D,\; f\in\barE$.
\end{remark}

\begin{remark}\upshape A sufficient condition for $\bfU\in \mathcal{D}(\bfeta)$ is given by
\begin{equation}\label{eqn:sharpening_of_equivalent_formulation_of_domain_of_attraction}
P(\bfU-1\le g)=1-\norm g_D+o(\norm g_\infty)\tag{\ref{eqn:equivalent_formulation_of_domain_of_attraction}'}
\end{equation}
as $\norm g_\infty\to 0$, uniformly for all $g\in\barE$ with
$\norm g_\infty\le 1$, i.e., for all $g$ in the unit ball of
$\barE$.
\end{remark}

\begin{example} Take $\bfU=\exp(\bfeta)$. Then $\bfU$ is a copula
process, and we obtain uniformly for $g\in\barE$ with $\norm
g_\infty\le 1-\varepsilon$ by using the approximation
$\log(1+x)=x+O\left(x^2\right)$ as $x\to 0$
\begin{align}\label{eqn:example_of_copula_process_satisfying_domain_of_attraction_condition}
P(\bfU-1\le g)&=P(\bfeta\le \log(1+g))\nonumber\\
&=\exp\left(-E\left(\sup_{t\in[0,1]}\left(\abs{\log(1+g(t))}Z_t\right)\right)\right)\nonumber\\
&=\exp\left(-E\left(\sup_{t\in[0,1]}\left(\abs{g(t)+O\left(g(t)^2\right)}
Z_t\right)\right)\right)\nonumber\\
&=\exp\left(-E\left(\sup_{t\in[0,1]}\left(\abs{g(t)} Z_t\right)\right)+O\left(\norm g_\infty^2\right)\right)\nonumber\\
&=1-\norm g_D+ O\left(\norm g_\infty^2\right),
\end{align}
 i.e., the copula process $\bfU=\exp(\bfeta)$ satisfies condition \eqref{eqn:sharpening_of_equivalent_formulation_of_domain_of_attraction}.
\end{example}

\subsection{Functional Domain of Attraction for arbitrary MSP}

We conclude from de Haan and Lin \cite{dehal01} that the
process $\bfY$ is in the domain of attraction (in the sense of weak convergence of probability measures on $C[0,1]$) of an MSP if, and
only if each $Y_t$ is in the domain of attraction of a univariate
extreme value distribution together with the condition that the
copula process converges in distribution to a standard MSP
$\bfeta$, that is
\[
\left(\max_{1\le i\le n}n(U^{(i)}_t-1)\right)_{t\in[0,1]}\to_D
\bm{\eta}
\]
in $C[0,1]$, where $\bfU^{(i)}$, $i\in\N$, are independent copies
of $\bfU$.  Note that the univariate margins determine the norming
constants, so  the norming functions can be chosen as the constant
functions $a_n=1/n$, $b_n=1$, $n\in\N$.

Following this idea, we give the defintion of functional domain of attraction of an arbitrary MSP in $C[0,1]$.

Note that for an arbitrary process $\bfY\in C[0,1]$ whose one-dimensional margins $Y_t$ have continuous distribution functions $F_t$ for each $t\in[0,1]$, we alway get a corresponding \emph{continuous} copula process $\bfU= \left(U_t\right)_{t\in[0,1]}:=\left(F_t(Y_t)\right)_{t\in[0,1]}$.

\begin{definition} Let $\bfY$ in $C[0,1]$ be a process with continuous one-dimensional margins and let  $\bfzeta$ be in $C[0,1]$ an MSP with arbitrary max-stable margins.
We say that $\bfY$ is in the \textit{functional domain of attraction} of $\bfzeta$, denoted by $\bfY\in \mathcal D(\bfzeta)$, if

\begin{enumerate}
	\item All one-dimensional margins of $\bfY$ are in the domain of attraction of the corresponding one-dimensional margin of $\bfzeta$, i.e. $Y_t\in \mathcal D(\zeta_t),\ t\in[0,1]$.
	\item $\bfU:=\left(F_t(Y_t)\right)_{t\in[0,1]} \in \mathcal D(\bfeta)$ for some standard max-stable process $\bfeta$.
\end{enumerate}

\end{definition}

\section{Functional GPD}\label{sec:functional_GPD}
\subsection{Generalized Pareto Processes}
A univariate GPD $W$ is simply given by $W(x)=1+\log(G(x))$, $G(x)\ge 1/e$, where $G$ is a univariate EVD. It was established by Pickands \cite{pick75} and Balkema and de Haan \cite{balh74} that the maximum of $n$ iid univariate observations, linearly standardized, converges in distribution to an EVD as $n$ increases if, and only if, the exceedances above an increasing threshold follow a generalized Pareto distribution (GPD). The multivariate analogon is due to Rootz\'{e}n and Tajvidi \cite{roott06}. It was observed by Buishand et al. \cite{buihz08} that a $d$-dimensional GPD $W$ with ultimately standard Pareto margins can be represented in its upper tail  as $W(\bfx)=P(U^{-1}\bfZ\le\bfx)$, $\bfx_0\le \bfx\le\bfzero\in\R^d$, where the rv $U$ is uniformly on $(0,1)$ distributed and independent of the rv $\bfZ=(Z_1,\dots,Z_d)$ with $0\le Z_i\le c$ for some $c\ge 1$ and $E(Z_i)=1$, $1\le i\le d$. We extend this approach to functional spaces. For a recent account of multivariate EVT and GPD we refer to Falk et al. \cite{fahure10}.

\begin{definition}
Let $U$ be a uniformly on $[0,1]$ distributed rv, which is independent of a generator process $\bfZ\in\bar C^+[0,1]$ with properties \eqref{eqn:properties_of_generator}. Then the stochastic process
\[
\bfY:=\frac 1U \bfZ\ \in\bar C^+[0,1].
\]
is called a \textit{GPD-process} (cf. Buishand et al.  \cite{buihz08}).
\end{definition}

The onedimensional margins $Y_t$ of $\bfY$ have ultimately standard Pareto tails:
\begin{align*}
P(Y_t\le x)&= P\left(\frac 1 x Z_t\le U\right)\\
&=\int_0^m P\left(\frac 1x z\le U\right)\,(P*Z_t)(dz)\\
&=1 - \frac 1 x \int_0^m z\,(P*Z_t)(dz)\\
&=1 - \frac 1 x E(Z_t)\\
&= 1- \frac 1 x,\qquad x\ge m,\,0\le t\le 1.
\end{align*}

Put $\bfV:=-1/\bfY$. Then, by Fubini's Theorem,
\begin{align*}
P\left(\bfV\le f\right)&=P\left(\sup_{t\in[0,1]}\left(\abs{f(t)}Z_t\right)\le U\right)\\
&=1-\int_0^1  P\left(\sup_{t\in[0,1]}\left(\abs{f(t)}Z_t\right)>u\right)\,du\\
&=1- E\left(\sup_{t\in[0,1]}\left(\abs{f(t)}Z_t\right)\right)\\
&=1-\norm f_D
\end{align*}
for all $f\in\barE$ with $\norm f_\infty\le 1/m$, i.e.,  $\bfV$ has the property that its distribution function is in its upper tail equal to
\begin{align}
W(f)&:=P\left(\bfV\le f\right)\notag\\
&=1-\norm f_D\notag\\
&=1+\log\left(\exp\left(-\norm f_D\right)\right)\notag\\
&=1+\log(G(f)),\qquad f\in \barE,\,\norm f_\infty\le 1/m \label{uppertail_funct_df_of_GPD},
\end{align}
where $G(f)=P(\bfeta\le f)$ is the functional df of the MSP $\bfeta$ with $D$-norm $\norm\cdot_D$ and generator $\bfZ$.

The preceding representation of the upper tail of a functional GPD in terms of $1+\log(G)$ is in complete accordance with the unit- and multivariate case (see, for example, Falk et al. \cite[Chapter 5]{fahure10}). We write $W=1+\log(G)$ in short notation and call $\bfV$ a GPD-process as well.

\begin{remark}Due to representation \eqref{uppertail_funct_df_of_GPD}, the GPD process $\bfV$ is clearly in the functional domain of attraction, in the sense of equation \eqref{cond:definition_of_domain_of_attraction}, of the standard MSP $\bfeta$ with $D$-norm $\norm\cdot_D$ and generator $\bfZ$; take $a_n=1/n$ and $b_n=0$.
\end{remark}

\begin{remark}
As mentioned by Buishand et. al \cite{buihz08}, the GPD-process $\bfY$ is in the domain of attraction of a \textit{simple} max-stable process $\bfxi$ in the sense of weak convergence on $C[0,1]$: for $\bfY_1,\bfY_2,\ldots$ independent copies of $\bfY$ we have
$$
\frac1n \max_{1\leq i\leq n} \bfY_i\ \to_D  \bfxi \qquad \textrm{in }C[0,1].
$$
\end{remark}

The following result is a functional version of the well-known fact  that the spectral df of a GPD random vector is equal to a uniform df in a neighborhood of 0.

\begin{lemma}
We have for $f\in\barE$  with $\norm f_\infty\le m$ and some $s_0<0$
\[
W_f(s):=P\left(\bfV\le s\abs f\right)=1+s\norm f_D,\qquad s_0\le s\le 0.
\]
\end{lemma}

Let $\bfU$ be a copula process. Then the following variant of Proposition \ref{prop:equivalences_to_FuDA} holds.

\begin{proposition} The property $\bfU\in \mathcal{D}(\bfeta)$ in the sense of condition \eqref{cond:definition_of_domain_of_attraction} is equivalent to
\begin{equation}\label{eqn:tail_equivalence_of_H_and_W}
\lim_{s\uparrow 0} \frac{1-H_f(s)}{1-W_f(s)}=1, \qquad f\in\barE,
\end{equation}
i.e., the spectral df $H_f(s)=P(\bfU-1\le s\abs f),\;s\le 0,$ of $\bfU-1$ is \textit{tail equivalent} with that of the GPD $W_f=1+\log(G_f)$, $G(\cdot)=\exp\left(-\norm{\cdot}_D\right)$.
\end{proposition}

We finish this section by defining a standard generalized Pareto process we are working with in the sequel.
\begin{definition}
A stochastic process $\bfV\in \barC$ is a \textit{standard generalized Pareto process} (GPP), if there exists a $D$-norm $\norm\cdot_D$ on $E[0,1]$ and some $c_0>0$ such that
 \[
 P(\bfV\le f)=1-\norm f_D
 \]
 for all $f\in\barE$ with $\norm f_\infty\le c_0$.
\end{definition}

\subsection{Spectral $\delta$-Neighborhood of a Standard GPP} \label{sec:Spectral_delta_Neighborhood_of_a_standard_GPP}

 Using the spectral decomposition of a stochastic process in $\barC$, we can easily extend the definition of a spectral $\delta$-neighborhood of a multivariate GPD as in Falk et al. \cite[Section 5.5]{fahure10} to the spectral $\delta$-neighborhood of a standard GPP.
 Denote by $\bar E^-_1[0,1]=\set{f\in\barE:\,\norm f_\infty=1}$ the unit sphere in $\barE$.

 \begin{definition}
 We say that a stochastic process $\bfY\in\barC$ belongs to the \textit{spectral $\delta$-neighborhood} of the GPP $\bfV$ for some $\delta \in(0,1]$, if we have uniformly for $f\in \bar E^-_1[0,1]$ the expansion
 \begin{align*}
 1-P(\bfY\le c f)&=(1-P(\bfV\le cf)\left(1+O\left(c^\delta\right)\right)\\
 &=\norm f_D \left(1+O\left(c^\delta\right)\right)
 \end{align*}
 as $c\downarrow 0$.
 \end{definition}
 A standard MSP is, for example, in the spectral $\delta$-neighborhood of the corresponding $GPP$ with $\delta=1$, see expansion \eqref{eqn:example_of_copula_process_satisfying_domain_of_attraction_condition}.

 The following result on the rate of convergence extends Theorem 5.5.5 in Falk et al. \cite{fahure10} on the rate of convergence of multivariate extremes to functional extremes.

 \begin{proposition}
 Let $\bfY$ be a stochastic process in $\barC$, $\bfV$ a standard GPP with $D$-norm $\norm\cdot_D$ and $\bfeta$ a corresponding standard MSP.
 \begin{itemize}
 \item[\textit{(i)}] Suppose that $\bfY$ is in the spectral $\delta$-neighborhood of $\bfV$ for some $\delta\in(0,1]$. Then we have
     \[
     \sup_{f\in\barE}\abs{P\left(\bfY\le \frac f n\right)^n-P(\bfeta\le f)}=O\left(n^{-\delta}\right).
     \]
\item[\textit{(ii)}] Suppose that $H_f(c)=P(\bfY\le c\abs f)$ is differentiable with respect to $c$ in a left neighborhood of $0$ for any $f\in\bar E^-_1[0,1]$, i.e., $h_f(c):=(\partial/\partial c)H_f(c)$ exists for $c\in(-\varepsilon,0)$ and any $f\in\bar E^-_1[0,1]$. Suppose, moreover, that $H_f$ satisfies the von Mises condition
    \[
    \frac{-ch_f(c)}{1-H_f(c)}=:1+r_f(c)\to_{c\uparrow 0} 1,\qquad f\in \bar E^-_1[0,1],
    \]
    with remainder term $r_f$ satisfying
    \[
    \sup_{f\in \bar E^-_1[0,1]}\abs{\int_c^0\frac{r_f(t)} t\,dt}\to_{c\uparrow 0} 0.
    \]
    If
    \[
    \sup_{f\in\barE}\abs{P\left(\bfY\le \frac f n\right)^n-P(\bfeta\le f)}=O\left(n^{-\delta}\right)
    \]
    for some $\delta\in(0,1]$, then $\bfY$ is in the spectral $\delta$-neighborhood of the GPP $\bfV$.
 \end{itemize}
 \end{proposition}

 \begin{proof}
 Note that
 \begin{align*}
 &\sup_{f\in\barE}\abs{P\left(\bfY\le \frac f n\right)^n-P(\bfeta\le f)}\\
 &= \sup_{f\in\barE}\abs{P\left(\bfY\le \frac{\norm f_\infty} n \frac f{\norm f_\infty}\right)^n- P\left(\bfeta\le \frac{\norm f_\infty} n \frac f{\norm f_\infty}\right)}\\
 &=\sup_{c < 0}\sup_{f\in\bar E^-_1[0,1]}\abs{P\left(\bfY\le c\abs f\right)^n- P\left(\bfeta\le c\abs f\right)}\\
  &=\sup_{f\in\bar E^-_1[0,1]} \sup_{c < 0}\abs{P\left(\bfY\le c\abs f\right)^n- P\left(\bfeta\le c\abs f\right)}.
 \end{align*}
 The assertion now follows by repeating the arguments in the proof of Theorem 1.1 in Falk and Reiss \cite{falr02}, where the bivariate case has been established.
   \end{proof}

\section*{Acknowledgment}
The second author gave an invited talk on an earlier version of
this paper at the  meeting 'On Some Current Research Topics in
Extreme Value Theory', organized by the FCT/MCTES Research Project
\textit{Spatial Extremes} and CEAUL on December 13, 2010,
University of Lisbon, Portugal. The authors are indebted to the
organizers, Ana Ferreira and Laurens de Haan, for the constructive
discussion from which the present version of the paper has
benefitted a lot.

\bibliographystyle{plainnat}
\bibliography{evt}

\end{document}